\documentclass[10pt]{article}
\usepackage{amsmath, amsthm, amssymb, }
\usepackage{float,epsfig}
\usepackage{subfigure}
\usepackage{pstricks,pst-node,pst-tree}
\usepackage{tikz-cd}
\usepackage{color,graphicx,float,lscape,enumerate}
\usepackage[colorlinks = true]{hyperref}
\usepackage{cleveref}
\usetikzlibrary{decorations.markings, arrows.meta}
\DeclareMathOperator{\lcm}{lcm}

\begin{document}

\newtheorem{theorem}{\bf Theorem}[section]
\newtheorem{proposition}[theorem]{\bf Proposition}
\newtheorem{definition}[theorem]{\bf Definition}
\newtheorem{corollary}[theorem]{\bf Corollary}
\newtheorem{example}[theorem]{\bf Example}
\newtheorem{exam}[theorem]{\bf Example}
\newtheorem{remark}[theorem]{\bf Remark}
\newtheorem{lemma}[theorem]{\bf Lemma}
\newcommand{\nrm}[1]{|\!|\!| {#1} |\!|\!|}

\newcommand{\ba}{\begin{array}}
\newcommand{\ea}{\end{array}}
\newcommand{\von}{\vskip 1ex}
\newcommand{\vone}{\vskip 2ex}
\newcommand{\vtwo}{\vskip 4ex}
\newcommand{\dm}[1]{ {\displaystyle{#1} } }

\newcommand{\be}{\begin{equation}}
\newcommand{\ee}{\end{equation}}
\newcommand{\beano}{\begin{eqnarray*}}
\newcommand{\eeano}{\end{eqnarray*}}
\newcommand{\inp}[2]{\langle {#1} ,\,{#2} \rangle}
\def\bmatrix#1{\left[ \begin{matrix} #1 \end{matrix} \right]}

\newcommand{\tb}[1]{\textcolor{blue}{ #1}}
\newcommand{\tm}[1]{\textcolor{magenta}{ #1}}
\newcommand{\tre}[1]{\textcolor{red}{ #1}}
\newcommand{\tg}[1]{\textcolor{olive}{ #1}}



\def \R{{\mathbb R}}
\def \C{{\mathbb C}}
\def \K{{\mathbb K}}
\def \M{{\mathcal M}}
\def \P{{\mathcal P}}
\def \pf{{\bf Proof: }}
\def \pds{{\mathrm {PDS}}}
\def \C{{\mathcal C}}
\def \T{{\mathcal T}}
\def \S{{\mathcal{O}}}

\title{On the trace-zero doubly stochastic matrices of order $5$}

\author{ Amrita Mandal\thanks{Corresponding author, Harish-Chandra Research Institute, Prayagraj, India, Email: mandalamrita55@gmail.com} \footnotemark[3] \thanks{Birla Institute of Technology Mesra, Ranchi, India} \, and \, Bibhas Adhikari\thanks{Department of Mathematics, IIT Kharagpur, India} \thanks{The author currently works at Fujitsu Research of America, Inc., California, USA, Email: bibhas@maths.iitkgp.ac.in}}
\date{}
 \maketitle
\thispagestyle{empty}

\small \noindent{\bf Abstract.} We propose a graph theoretic approach to determine the trace of the product of two permutation matrices through a weighted digraph representation for a pair of permutation matrices. Consequently, we derive trace-zero doubly stochastic (DS) matrices of order $5$ whose $k$-th power is also a trace-zero DS matrix for $k\in\{2,3,4,5\}$. Then, we determine necessary conditions for the coefficients of a generic polynomial of degree $5$ to be realizable as the characteristic polynomial of a trace-zero DS matrix of order $5$. Finally, we approximate the eigenvalue region of trace-zero DS matrices of order $5.$  


\vone \noindent{\bf Keywords.} Permutation matrix, doubly stochastic matrix, inverse eigenvalue problem
	
 \vone\noindent{\bf AMS subject classification(2000):} 15A18, 15B51.
		
\section{Introduction}\label{sec:1}

Given a collection $\Lambda$ of $n$ scalars and a set of structured matrices $\mathbb{S}\subseteq\K^{n\times n}$, the \textit{inverse eigenvalue problem} (IEP), also known as \textit{eigenvalue realizibility problem}, is concerned with finding  $A\in\mathbb{S}$ such that $\Lambda=\sigma(A),$ where $\sigma(A)$ denotes the (multi)set of eigenvalues of $A,$ and $\K$ is the field of real or complex numbers \cite{Johnson2018}. There has been a paramount interest for solving IEP when $\mathbb{S}$ is the set of nonnegative matrices, stochastic matrices, doubly stochastic (DS) matrices or particular subsets of such matrices \cite{johnson1981row, Mourad, Lei, Johnson2018, Nader, mandal2019eigenvalue}. A vast literature is available for these problems when $n\leq 4,$ whereas only a little is known for matrices of order $5$ or more except for some sets of doubly structured matrices. 
For instance, IEP for symmetric nonnegative matrices of order $5$ has been investigated recently in \cite{marijuan2023note}. 
Besides, necessary and sufficient conditions for the eigenvalue realizability problem for trace-zero nonnegative matrices of order $5$ are discussed \cite{trace0,niepebl}. 

A more general problem in the context of IEP is to determine the eigenvalue region of a given $\mathbb{S}\subseteq \K^{n\times n},$ i.e. characterizing $\bigcup_{A\in \mathbb{S}}\sigma(A).$ Indeed, let  $\Omega_n$ and $\Omega_n^0$ denote the set of all DS matrices and trace-zero DS matrices of order $n$ respectively, and $\omega_n$ and $\omega_n^0$ denote the eigenvalue region of $\Omega_n$ and $\Omega_n^0$ respectively. Then it is shown that $\cup_{j=2}^n \Pi_j \subseteq \omega_n, n\geq 2,$ and  $\omega_n =\cup_{j=2}^n \Pi_j$ for $2\leq n\leq 4,$ where $\Pi_j$ denotes the convex hull of the $j^{th}$ roots of unity \cite{ Perfect1965,Mashreghi2007,Levick2015,harlev}. However, characterization of $\Omega_n$ is a long standing open problem for $n\geq 5.$ For trace-zero DS matrices, it is proved that $\Pi_n^0 \subseteq \omega_n^0$ and $\omega_3^0=\Pi_3^0,\,\omega_4^0=\Pi_4^0 \cup [-1,1],$ where $\Pi_j^0$ is the union of $1$ and convex hull of the $j^{th}$ roots of unity except the point $1$ \cite{Perfect1965,Benvenuti2018}. In that context, we mention that IEP is described for DS matrices of orders up to $3$ in \cite{Perfect1965}, and algebraic relations among the coefficients of the characteristic polynomial of a matrix in $\Omega_4^0$ are given in \cite{Benvenuti2018}.


In this paper, we consider $\Omega^0_5$ and develop a graph theoretic representation of pairs of permutation matrices of order $5$ that enables us to approximate $\omega_5^0.$ This is achieved by considering the trace-zero DS matrices as convex combinations of trace-zero permutation matrices and by finding the trace of the product of two trace-zero permutation matrices of order $5$. In this process, we determine the matrices $A\in \Omega_5^0$ such that $A^k\in \Omega_5^0,$ $k=2,3,4,5$ and we show that any such matrix can be expressed as convex combinations of at most three permutation matrices. Consequently, we determine algebraic relations between coefficients of a generic polynomial of degree $5$ so that the polynomial represents the characteristic polynomial of some $A\in\Omega^0_5.$ Finally, we approximate $\omega_5^0.$

The rest of the paper is organized as follows. In Section \ref{sec:preli}, we develop a method to estimate the trace of the product of two permutation matrices with certain cycle types by defining a weighted digraph associated with the pair of permutation matrices. Section \ref{sec:3} describes trace-zero DS matrices of order $5$ whose $k$-th power are also trace-zero matrices, when $k \in \{2,3,4,5\}$. In Section \ref{sec:4}, we determine necessary conditions of the coefficients of a generic polynomial of degree $5$ to be realized as the characteristic polynomial of a trace-zero DS matrix of order $5$.

{\bf Notation:} The symmetric group on $n$ elements is denoted as $S_n.$ 
If $\pi \in S_n$ is a product of $r$ disjoint cycles of lengths $k_1, k_2,\ldots, k_r$, then we say $\pi$ is a permutation of cycle type $(k_1, k_2, \ldots, k_r).$ The group of permutation matrices of order $n$ is denoted as $\P_n.$ The set of permutations with no $1$-cycle in its cycle type is denoted as $S_n^0,$ and the corresponding set of permutation matrices with trace zero is denoted as $\P_n^0.$ $I_n$ denotes the identity matrix of order $n,$ and  $J$ denotes the all-ones matrix whose dimension will be clear from the context. The trace of the matrix $A$ is denoted as $tr(A).$


\section{The trace of the product of two permutation matrices} \label{sec:preli}

In this section, we develop a method to determine the trace of the product of two permutation matrices based on its cycle types corresponding to a weighted digraph representation of a permutation matrix. For any permutation $\pi \in S_n$, the symmetric group of order $n$, the corresponding permutation matrix $P_{\pi}=(P_{ij})$ is defined as $P_{ij}=1$ if $j=\pi(i),$ and $0$ otherwise. First, we review the notion of isomorphic weighted digraphs.     


 A weighted digraph is an ordered triple $G=(V, E, W)$ where $V$ denotes the vertex set, $E\subseteq V \times V$ is the edge set, and $W: E\rightarrow \mathbb{R^+}$ assigns a weight to each edge of the digraph. An edge $(v,v)\in E$ is called a loop at the vertex $v.$ If $(u,v)\in E$ then we denote $\overrightarrow{uv}$ to mean that the edge is directed from $u$ to $v.$ Two weighted digraphs  $G_1\equiv G_1(V_1,E_1,W_1)$ and $G_2 \equiv G_2(V_2,E_2,W_2)$ are said to be isomorphic if there exists a bijective function  $f:V_1 \rightarrow V_2$ such that
\begin{enumerate}
    \item $\overrightarrow{uv}\in E_1$ if and only if $\overrightarrow{f(u)f(v)}\in E_2.$
    \item $W_1(\overrightarrow{uv})=W_2(\overrightarrow{f(u)f(v)})$ for every $\overrightarrow{uv}\in E_1.$
\end{enumerate}
Then clearly $f\in S_n$ if $V_1 = V_2$
and $|V_i|=n, i=1,2.$


A weighted digraph $G_\pi=(V_\pi, E_\pi, W_\pi)$ corresponding to a permutation $\pi\in S_n$ is defined as follows. Set $V_\pi=\{1, \hdots,n\},$ and $\overrightarrow{ij}\in E_\pi$ if $j=\pi(i),$ $i,j\in V_\pi.$ Then observe that $G_\pi$ is a union of disjoint directed cycles and loops. Each cycle-subgraph in $G_\pi$ corresponds to a cycle in the permutation $\pi.$ The weight function $W_\pi:E\rightarrow \R^+$ is a constant function, symbolically denoted as $W_\pi(e)=w_\pi$, $e\in E$ for brevity.

Then we define a weighted digraph corresponding to a pair of permutations $\pi,\tau$ of different weights, i.e. with $w_\pi \neq w_\tau$, as the union of $G_\pi=(V_\pi, E_\pi, W_\pi)$ and $G_\tau=(V_\tau, E_\tau, W_\tau)$, denoted by $G_{\{\pi,\tau\}}=(V_{\{\pi,\tau\}}, E_{\{\pi,\tau\}}, W_{\{\pi,\tau\}})$ where\begin{enumerate}
    \item $V_{\{\pi,\tau\}}=V_\pi=V_\tau=\{1,\hdots,n\}$
    \item $E_{\{\pi,\tau\}}=E_\pi\cup E_{\tau}$
    \item $W_{\{\pi,\tau\}}: E_{\{\pi,\tau\}}\rightarrow \mathbb{R}^+$ such that $W_{\{\pi,\tau\}}(e)=w_\pi$ if $e\in E_\pi\setminus E_{\tau},$ $W_{\{\pi,\tau\}}(e)=w_\tau$ if $e\in E_{\tau}\setminus E_\pi,$ and $W_{\{\pi,\tau\}}(e)=w_\pi+w_{\tau}$ if $e\in E_\pi\cap E_{\tau}.$
\end{enumerate} 

Two such digraphs $G_{\{\pi_1,\tau_1\}}$ and $G_{\{\pi_2,\tau_2\}}$ are called isomorphic if they are isomorphic as weighted digraphs. The isomorphic relation on the set $\{G_{\{\sigma,\tau\}} : \sigma, \tau\in S_n\}$ is an equivalence relation, and we call an equivalence class defined by this relation as \textit{equivalence graph class}.

Also recall that the conjugate of an element $a$ of a group $\mathbb{G}$ by an element $b\in \mathbb{G}$ is an element $bab^{-1}\in\mathbb{G}.$ The relation $\sim$ on $\mathbb{G},$ defined by $a \sim b,$ if $a$ is a conjugate of $b$ is an equivalence relation. The equivalence
classes under this relation are called the conjugacy classes of $\mathbb{G}.$ For the symmetric group $S_n,$ the conjugate $\tau \pi\tau^{-1}=(\tau(a_1) \, \tau(a_2) \, \hdots \, \tau(a_k))$ for a $k$-cycle $(a_1\, a_2\, \hdots \, a_k)=\pi\in S_n.$ The conjugacy classes in the permutation group $S_n$ are determined by the cycle type of the permutations. Indeed, two elements of $S_n$ are in the same conjugacy class if and only if they have the same cycle type. Then, we have the following proposition.

\begin{proposition}\label{prop}
 Let $ G_1\equiv G_{\{\sigma_1,\tau_1\}}$ and $G_2\equiv G_{\{\sigma_2,\tau_2\}}$,  $\sigma_i,\tau_i \in S_n,i=1,2.$ Suppose $w_{\sigma_1}=w_{\sigma_2}=w_\sigma$ and $w_{\tau_1}=w_{\tau_2}=w_\tau.$ Let $f: \{1,\hdots,n\}\rightarrow \{1,\hdots,n\}$ be a bijection. Then $f$ defines an isomorphism from $G_1$ to $G_2$ if and only if  $f\sigma_1 f^{-1}=\sigma_2$ and $f\tau_1 f^{-1}=\tau_2.$ 
 \end{proposition} 
 
\begin{proof}
Let $f$ define an isomorphism from $G_1$ to $G_2.$  Let $i\in\{1,\hdots,n\}.$ Then suppose $\sigma_2(i)=j\in V,$ that is $\overrightarrow{ij}\in E_{\sigma_2}.$ Now, if $f^{-1}(i)=k\in V$ then there exists $l\in V$ such that either $\overrightarrow{kl}\in E_{\sigma_1}$ or $\overrightarrow{kl}\in E_{\tau_1}$ since $G_{\{\sigma_1,\tau_1\}}$ and $G_{\{\sigma_2,\tau_2\}}$ are isomorphic under the isomorphism $f,$ and $f(l)=j$ with $W_{\{\sigma_2,\tau_2\}}(\overrightarrow{ij})=W_{\{\sigma_1,\tau_1\}}(\overrightarrow{kl}).$ There can exist only one such vertex $l$ due to the fact that the vertices in digraphs corresponding to $\sigma_x,\tau_x, x=1,2,$ have in-degree or out-degree $1.$  If $\overrightarrow{kl}\in E_{\sigma_1}$ then we are done since $f \sigma_1 f^{-1}(i)=\sigma_2(i).$ Otherwise,  $\overrightarrow{kl}\in E_{\tau_1}$ and $\sigma_1(k)=m\neq l$. However, this is a contradiction since $f$ preserves the weights. Similarly we can show $f \tau_1 f^{-1}(i)=\tau_2(i).$
 
To prove the converse, let $f\sigma_1 f^{-1}=\sigma_2$ and $f\tau_1 f^{-1}=\tau_2.$ Now let $\overrightarrow{ij}$ be an edge in $G_1.$ Then either $j=\sigma_1(i)$ or $j=\tau_1(i)$ or both hold together so that one or both of $f(j)=f\sigma_1(i)=\sigma_2 f(i)$ and $f(j)=f\tau_1(i)=\tau_2 f(i)$ hold. Thus $\overrightarrow{f(i)f(j)}\in E_{\{\sigma_2,\tau_2\}}.$ Similarly if $\overrightarrow{kl}$ is an edge in $G_2,$ then either $l=\sigma_2(k)$ or $l=\tau_2(k)$ or both hold. Also, there exist $i,j\in V_{\{\sigma_1,\tau_1\}}$ with $i\neq j$ such that $f(i)=k$ and $f(j)=l,$ since $f$ is a bijection from $G_1$ to $G_2.$ Hence, one or both of $f\sigma_1(i)=\sigma_2 f(i)=\sigma_2(k)=l=f(j)$ and $f\tau_1(i)=\tau_2 f(i)=\tau_2(k)=l=f(j)$ are satisfied. This further implies that either $\sigma_1(i)=j$ or $\tau_1(i)=j$ or $\sigma_1(i)=\tau_1(i)=j,$ i.e. $\overrightarrow{ij}\in E_{\{\sigma_1,\tau_1\}}.$ Clearly under the assumptions that $w_{\sigma_1}=w_{\sigma_2}=w_\sigma$ and $w_{\tau_1}=w_{\tau_2}=w_\tau,$ $W_{\{\sigma_1,\tau_1\}}(\overrightarrow{ij})=W_{\{\sigma_2,\tau_2\}}(\overrightarrow{f(i)f(j)}),$ for $\overrightarrow{ij}\in E_{\{\sigma_1,\tau_1\}}.$ Hence $f$ is an isomorphism from $G_1$ to $G_2.$ 
\end{proof}

Let $P_{\alpha}=P_{\pi} P_{\tau},$ $P_\pi,P_\tau\in \P_n.$   Then the $ij$th entry of $P_\alpha=P_{\alpha}(i,j)=1$ implies that there exists exactly one $k,$ $1\leq k\leq n$ such that $P_{\pi}(i,k)=P_{\tau}(k,j)=1.$ Hence $j=\tau(k)$ and $k=\pi(i),$ i.e. $j=\tau\pi(i),$ thus $\alpha=\tau\pi.$
If $ G_1\equiv G_{\{\pi_1,\tau_1\}}$ and $G_2\equiv G_{\{\pi_2,\tau_2\}}$ are isomorphic digraphs under the isomorphism $f:G_1\rightarrow G_2$ with $w_{\pi_1}=w_{\pi_2},$ $w_{\tau_1}=w_{\tau_2},$ then from Proposition \ref{prop} we obtain
\begin{itemize}
   \item[(a)]  $P_f^TP_{\pi_1}P_f=P_{\pi_2}, \,\, P_f^TP_{\tau_1}P_f=P_{\tau_2}$ \item[(b)]  $P_f^TP_{\tau_1}P_{\pi_1}P_f=P_{\tau_2}P_{\pi_2}, \,\, P_f^TP_{\pi_1} P_{\tau_1}P_f=P_{\pi_2}P_{\tau_2}.$
\end{itemize} Clearly, the permutation matrices $P_{\tau_1}P_{\pi_1}$ and $P_{\tau_2}P_{\pi_2}$ are conjugate to each other, so that they correspond to permutations of the same cycle type. Similarly, the permutation matrices $P_{\pi_1}P_{\tau_1}$ and $P_{\pi_2}P_{\tau_2}$ correspond to permutations of the same cycle type.






In what follows, we derive the number of equivalence graph classes in $\{G_{\{\pi,\tau\}}\tm{: \pi}, \tau\in S_n\}.$ The equivalence graph class corresponding to a pair $(\pi_1,\tau_1)$ is defined as
$$[G_{\{\pi_1,\tau_1\}}] := \left\{G_{\{\pi,\tau\}} : \mbox{there exists} \, f\in S_n \, \mbox{for which} \, f\pi f^{-1}=\pi_1, f \tau f^{-1}=\tau_1\right\}.$$

We mention that, in \cite{harlev} the authors have considered the problem of identifying \textit{non-isomorphic} permutation pairs in $S_n\times S_n,$ by defining the action of $S_n$ on the group $S_n \times S_n$ such that 
 $p\in S_n$ acts on $(\pi,\beta)\in S_n\times S_n$ by $(\pi,\beta)^p=(p\pi p^{-1},p\beta p^{-1}).$ Then the set of pairs $(\pi_1,\beta_1)$ and $(\pi_2,\beta_2)$ belong to the same equivalence class if they are related by $(\pi_1,\beta_1)^p=(\pi_2,\beta_2).$ Thus, the representative pairs have been derived from the orbits of the action with the help of a computational simulation for $n\leq 13$. Here, we have taken a different approach by means of the idea of isomorphism of weighted directed graphs for the determination of non-isomorphic permutation pairs in $S_n\times S_n.$ We provide an analytical approach to count the number of distinct equivalence classes for pairs of permutations that have specific cycle types and pave the way for several possible cases.   
 
First, we prove the following lemmas. 



\begin{lemma} \label{lemma:full cycle}
 Let $\pi\in S_n$ be an $n$-cycle. Let $f\in S_n,$ then $f\pi f^{-1}=\pi$ if and only if $f\in \langle \pi\rangle,$ where $\langle \pi\rangle=\{\pi^k: k \in \mathbb{Z}\}$ is the cyclic subgroup of $S_n$ generated by $\pi.$ 
\end{lemma}

\begin{proof} Let $\pi=(a_1a_2\ldots a_n),a_i\in\{1,\ldots,n\}.$
Then $f\pi f^{-1}=\pi$ implies $f\pi(a_i)=\pi f(a_i)$ for all $a_i,i=1,\ldots,n.$ Thus $f(a_{i+1})=\pi f(a_i)= 
\pi^2 f(a_{i-1})=\ldots=\pi^i f(a_1)$ for $i=1,\ldots,n(\mbox{mod}\, n).$ Now, $\pi(a_i)=a_{i+1}$ implies $\pi^{k}(a_i)=a_{i+k}.$ So that if $f(a_1)=a_m$ for some $m\in\{1,2,\ldots,n\},$ then $f(a_{i+1})=\pi^i f(a_1)=\pi^i a_m=a_{m+i}=\pi^{m-1}(a_{i+1})$ for all $a_{i+1}, i=1,\ldots,n.$ Hence $f$ is of the form $\pi^{m-1}$ for some $m\in\{1,2,\ldots,n\},$ which proves $f\pi f^{-1}=\pi$ yields $f\in \langle \pi\rangle.$ The converse is obvious.\end{proof}

Note that, $\{f\in S_n| f\pi f^{-1}=\pi\}$ is the stabilizer of $\pi\in S_n$ under the conjugate action of $S_n$ on itself. 

\begin{lemma} \label{lemma:n=k_1k_2}
 Let $\pi=\pi_1\pi_2\in S_n$ be of cycle type $(k_1,k_2),$ where $\pi_1$ and $\pi_2$ are disjoint cycles of length $k_1$ and $k_2$ respectively. Let $f\in S_n.$ Then $f\pi f^{-1}=\pi$ if and only if $f\in \langle \pi_1,\pi_2\rangle,$ where $\langle \pi_1,\pi_2\rangle=\{\pi_1^{p}\pi_2^{q}: p,q \in \mathbb{Z}\}$ is the subgroup of $S_n$ generated by $\pi_1$ and $\pi_2.$ 
\end{lemma}

\begin{proof}
Let $\pi_1=(a_1\ldots a_{k_1})$ and $\pi_2=(a_{k_1+1}\ldots a_{k_1+k_2}),a_i\in\{1,\ldots,n\}.$ Then for $1\leq i \leq k_1,$
$\pi^{k}(a_i)=\pi_1^{k}(a_i)=a_{i+k}=a_j$ where $i+k\equiv j(\mbox{mod}\,k_1)$ and for $1\leq i \leq k_2,$ $\pi^{k}(a_{k_1+i})=\pi_2^{k}(a_{k_1+i})=a_{k_1+i+k}=a_{k_1+j}$ where $i+k\equiv j(\mbox{mod}\,k_2).$ Now $f\pi f^{-1}=\pi$ implies $f\pi(a_i)=\pi f(a_i)$ for $i=1,\ldots,n.$ Thus, $f(a_{i+1})=\pi f(a_i)=\ldots=\pi^i f(a_1)$ whenever $1\leq i \leq k_1$ and $f(a_{k_1+i+1})=\pi f(a_{k_1+i})=\ldots=\pi^i f(a_{k_1+1})$ whenever $1\leq i \leq k_2.$ Now we have the following cases under consideration.

Case $1$: Let $f(a_1)=a_p$ and $f(a_{k_1+1})=a_{k_1+q}$ for some $p\in\{1,\ldots,k_1\},q \in \{1,\ldots,k_2\}.$ Then $f(a_{i})=\pi^{i-1} f(a_1)=\pi^{i-1} (a_p)=\pi_1^{i-1}(a_p)=a_{p+i-1}=\pi_1^{p-1}(a_{i})=\pi_1^{p-1}\pi_2^{q-1}(a_{i})$ for all $a_{i}, i=1,\ldots,k_1.$ Similarly, we obtain $f(a_{k_1+i})=\pi^{i-1} f(a_{k_1+1})=\pi_2^{i-1}(a_{k_1+q})=a_{k_1+q+i-1}=\pi_2^{q-1}(a_{k_1+i})=\pi_1^{p-1}\pi_2^{q-1}(a_{k_1+i})$ for all $a_{k_1+i}, i=1,\ldots,k_2.$ Hence $f$ is of the form $\pi_1^{p-1}\pi_2^{q-1}$ for some $p\in\{1,\ldots,k_1\}$ and $q\in\{1,\ldots,k_2\}.$ Further $\pi_1^{k_1}=id,\pi_2^{k_2}=id,$ and hence we say $f\in \langle \pi_1,\pi_2\rangle.$ 

Case $2$: Let $f(a_1)=a_p$ and $f(a_{k_1+1})=a_{q},$ for some $p,q\in\{1,\ldots,k_1\}.$ Clearly $p\neq q.$ If $q>p,$ then $a_q=\pi_1^{q-p}(a_p)=f(a_{k_1+1})$ implies $\pi_1^{q-p} f(a_1)=f(a_{k_1+1}),$ i.e. $f(a_{q-p+1})=f(a_{k_1+1}).$ But $q-p<k_1$ and hence $f(a_{q-p+1})=f(a_{k_1+1})$ is a contradiction, since $f\in S_n$ is a bijective function. Otherwise if $q<p$ then we write $a_p=\pi_1^{p-q}(a_q)=f(a_1),$
which implies $\pi_1^{p-q} f(a_{k_1+1})=f(a_{1}),$ i.e. $f(a_{k_1+1})=f(a_{1}).$ Which is a contradiction.

Similarly we arrive at a contradiction if we choose $f(a_1)=a_{k_1+p}$ and $f(a_{k_1+1})=a_{k_1+q},$ for some $p,q\in\{1,\ldots,k_2\}.$ 

Case $3$: Let $f(a_1)=a_{k_1+p}$ and $f(a_{k_1+1})=a_{q},$ for some $p\in\{1,\ldots,k_2\}$ and $q\in \{1,\ldots,k_1\}.$ Then similarly as case $2,$ we obtain $a_q=\pi_1^{q-p-k_1} f(a_{1})=f(a_{k_1+1}),$ i.e. $f(a_{1})=\pi_1^{k_1-q+p}f(a_{k_1+1})=f(a_{k_1+1}),$ which contradicts.

Thus from the above cases we get $f\pi f^{-1}=\pi$ implies $f\in \langle \pi_1,\pi_2\rangle.$ The converse is obvious.
\end{proof}



\begin{lemma} \label{lemma:n prime}
 Let $n$ be a prime number. Then every non identity element in  $\langle \pi\rangle$ is an $n$-cycle whenever $\pi\in S_n$ is an $n$-cycle. 
\end{lemma}

\begin{proof} Since $n$ is prime $\gcd(n,k)=1$ for $1\leq k\leq n-1.$ Hence for $1\leq k\leq n-1$ the order of $\pi^k\in\langle \pi\rangle$ is $n,$  which is same as the order of $\pi.$ If possible, for a $k\in \{1,\ldots,n-1\},$ let $\pi^k$ be of cycle type $(m_1,m_2,\ldots,m_s),$ where $1\leq m_i<n$ for $i=1,\ldots,s.$ Hence $\lcm(m_1,\ldots,m_s)=n,$ which is possible only when one of $m_i>1$ and the others are $1.$ It follows that $\pi^k$ has order $m_i<n,$ a contradiction. Hence $\pi^k$ can be only a $n$-cycle. Thus, the proof follows.
\end{proof}

\begin{remark} \label{remark:k_1k_2prime}
\begin{enumerate}
    \item If $k_1$ and $k_2$ are coprime, i.e. $\gcd(k_1,k_2)=1$ in Lemma \ref{lemma:n=k_1k_2}, then $f\pi f^{-1}=\pi$ if and only if $f \in \langle \pi\rangle.$  
\item Let $k_1$ and $k_2$ be prime numbers. Then as a consequence of Lemma \ref{lemma:n prime}, it follows that the non identity element in  $\langle \pi_1,\pi_2\rangle$ is either of the three cycle types $(k_1,k_2),(k_1,1,\ldots,1)$ and $(1,\ldots,1,k_2),$ when $\pi=\pi_1\pi_2\in S_n$ is of cycle type $(k_1, k_2).$ Clearly, $k_1-1,k_2-1$ and $k_1k_2-k_1-k_2+1=(k_1-1)(k_2-1)$ are the numbers of permutations in $\langle \pi_1,\pi_2\rangle$ of cycle types $(k_1,1,\ldots,1), (1,\ldots,1,k_2)$ and $(k_1,k_2),$ respectively.
 \end{enumerate}
\end{remark}

\begin{lemma} \label{lemma:k+1..+1}
 Let $\pi\in S_n$ be of cycle type $(k,1,\ldots,1), k\in \mathbb{N.}$ Then $f\pi f^{-1}=\pi$ if and only if $f\in \langle \pi\rangle \ast S_{n-k},$ where $\langle \pi\rangle \ast S_{n-k}=\{\pi_1.\pi_2: \pi_1\in \langle \pi\rangle\,\mbox{and}\, \pi_2  \in S_{n-k}\}$ and $S_{n-k}$ is the symmetric group of degree $(n-k)$ on the set $\{k+1,\ldots,n\}.$
\end{lemma}
\begin{proof}
It is easy to check that $\langle \pi\rangle \ast S_{n-k}$ is a subgroup of $S_n.$
Let $\pi=(a_1\ldots a_{k}),a_i\in\{1,\ldots,k\}.$ Then for $1\leq i \leq k,$
$\pi^{l}(a_i)=a_{i+l}=a_j$ where $i+l\equiv j(\mbox{mod}\,k)$ and $\pi^{l}(a_i)=a_i$ for $i=k+1,\ldots,n.$ Now $f\pi f^{-1}=\pi$ implies $f\pi(a_i)=\pi f(a_i)$ for $i=1,\ldots,n.$ Thus $f(a_{i+1})=\pi^i f(a_1)$ whenever $1\leq i \leq k.$ Now we have the following cases under consideration.

Case $1$: Let $f(a_1)=a_p$ for some $p\in\{1,\ldots,k\}.$ Then $f(a_{i})=\pi^{i-1} f(a_1)=\pi^{i-1} (a_p)=a_{p+i-1}=\pi^{p-1}(a_{i})$ for all $a_{i}, i=1,\ldots,k.$ Clearly $f$ can be any permutation when restricted to the set $\{k+1,\ldots,n\}.$ 
Hence $f$ is of the form $\pi^{p-1}\sigma$ for some $p\in\{1,\ldots,k\}$ and $\sigma\in S_{n-k}.$

Case $2$: Let $f(a_1)=a_p$ for some $p\in\{k+1,\ldots,n\}.$ Then $f(a_{i})=\pi^{i-1} f(a_1)=\pi^{i-1} (a_p)=a_{p}$ for all $a_{i}, i=1,\ldots,k.$ Which is not possible.

Thus $f\pi f^{-1}=\pi$ yields $\langle \pi\rangle \ast S_{n-k}.$ The proof for the converse follows directly.
\end{proof}

Let $\beta\in S_{n}.$ Let $k_{1},k_{2},\ldots,k_{s}$ be distinct integers which appear as lengths of cycles in the cycle type of $\beta,$ and $m_{i}$ is the number of cycles of length $k_{i}$ in $\beta,$   $i=1,2,\hdots,s.$ 
Then, the number of conjugates of $\beta,$  $N^{\beta}$, i.e. the number of permutations in $S_n$ with same cycle type as $\beta,$ is given by \cite{dummit1991}
$$N^{\beta}=\frac{n!}{\left(m_1!k_1^{m_1}\right)\left(m_2!k_2^{m_2}\right)\ldots\left(m_s!k_s^{m_s}\right)}.$$   

Now we proceed to determine the number of equivalence graph classes in the set of all digraphs $G_{\{\pi,\tau\}}$ when $\pi,\tau$ have specific cycle types. For brevity we write cycle type of a permutation $\pi$ as $ct(\pi).$ Note that, if 
$G_{\{\pi_1,\tau_1\}}$ and  $G_{\{\pi_1,\tau\}}$ are  isomorphic digraphs then clearly by Proposition \ref{prop} $f\pi_1 f^{-1}=\pi_1$ and $f\tau_1f^{-1}=\tau$ hold simultaneously for some $f\in S_n.$ Let $\pi\in S_n$ be such that $ct(\pi)=ct(\pi_1)$ and hence $\pi=g\pi_1g^{-1}$ for some $g\in S_n.$ Then $\pi=gf\pi_1f^{-1}g^{-1}$ and $g\tau g^{-1}=gf \tau_1 f^{-1}g^{-1}.$ Therefore, $ G_{\{\pi,g\tau g^{-1}\}},  G_{\{\pi_1,\tau_1\}}$ and $G_{\{\pi_1,\tau\}}$ are isomorphic. Hence, the following is true for a pair $(\pi_1, \tau_1)\in S_n\times S_n:$
\begin{eqnarray} \bigcup_{\pi,\tau \in S_n, ct(\pi)=ct(\pi_1), ct(\tau)=ct(\tau_1)} [G_{\{\pi,\tau\}}] &=& \left\{G_{\{\pi, \tau\}} : \pi \,\mbox{and}\, \tau \, \mbox{have same cycle type as} \, \pi_1, \tau_1 \,\mbox{respectively}\right\} \nonumber \\ &=& \bigcup_{\pi\in S_n, \,ct(\pi)=ct(\pi_1)}\left\{G_{\{\pi, \tau\}} : \tau \, \mbox{has same cycle type as} \, \tau_1\right\}.\label{eqn:ct}\end{eqnarray} Thus we conclude that we first need to identify the digraphs in equation (\ref{eqn:ct}) in order to determine the equivalence classes of $\{G_{\{\pi,\tau\}}\tm{: \pi},\tau\in S_n\}.$ 

First, we have the following propositions.

\begin{proposition} \label{pro:both n prime}
Let $n>2$ be a prime number. Then the number of equivalence graph classes in $\{G_{\{\pi,\beta\}}:\beta,\pi\in S_n\,\mbox{are both $n$-cycles}\}$ is $n+k-1,$ where $k=\frac{(n-1)!-(n-1)}{n}.$
\end{proposition}
\begin{proof}
The number of $n$-cycles in $S_n$ is $(n-1)!.$
Let $\pi\in S_n$ be a $n$-cycle. By Lemma \ref{lemma:full cycle}, $f \pi f^{-1}=\pi$ if and only if $f\in \langle \pi \rangle.$ 
 For $\beta \in S_n$ let us denote the multi set $\{f\beta f^{-1}:f\in \langle \pi \rangle\}$ as $\S_{\beta},$ which is the same as the orbit of $\beta$ when $\langle \pi \rangle$ acts on itself by conjugation. Then by Proposition \ref{prop}, $\{G_{\{\pi,\beta'\}}:\beta'\in \S_{\beta}\}$ is the collection of all digraphs isomorphic to $G_{\{\pi,\beta\}}.$ 
 
First consider $\beta \in \langle \pi \rangle$ as $n$-cycle. Then obviously $f\beta f^{-1}=\beta$ for any $f \in \langle \pi \rangle,$ so that $\S_{\beta}=\{\beta\}.$
Hence for $\beta,\beta'\in \langle \pi \rangle \setminus \{id\}$ with $\beta\neq\beta',$
$\S_{\beta}\cap \S_{\beta'}=\emptyset.$
Thus we have $(n-1)$ non isomorphic digraphs $G_{\{\pi,\beta_0\}},G_{\{\pi,\beta_1\}},\ldots,G_{\{\pi,\beta_{n-2}\}},$ where $\beta_i=\pi^{i+1},$  $i=0,\ldots,n-2.$

Now we choose an $n$-cycle $\beta \not\in \langle \pi \rangle.$ Clearly there can be $(n-1)!-(n-1)$ of such $\beta.$ 
By Lemma \ref{lemma:n prime}, $f\in \langle \pi \rangle$ is a $n$-cycle and 
$ \langle f \rangle=\langle \pi \rangle.$
If possible, for $f_1,f_2\in\langle \pi \rangle$ and $f_1\neq f_2,$ let $f_1\beta f_1^{-1}=f_2\beta f_2^{-1}.$ Which further implies $(f_1^{-1}f_2)\beta (f_1^{-1}f_2)^{-1}=\beta$ where $id\neq f_1^{-1}f_2\in \langle \pi \rangle$ and hence $\beta\in \langle f_1^{-1}f_2 \rangle =\langle \pi \rangle$ by Lemma \ref{lemma:full cycle}. This contradiction shows that $\S_{\beta}$ contains $n$ distinct elements, whenever $\beta \not\in \langle \pi \rangle$ and $ct(\beta)=ct(\pi).$ Thus, by choosing $\beta=\beta_{n-1}\not \in \langle \pi \rangle,$  we have the digraph $G_{\{\pi,\beta_{n-1}\}},$ which is non isomorphic to $G_{\{\pi,\beta_i\}},i=0,\ldots,n-2.$ 
Next, consider $\beta=\beta_n\not\in\langle \pi \rangle\cup\S_{\beta_{n-1}}.$ Then similarly 
$\S_{\beta_n}$ contains all the distinct elements and $\S_{\beta_n}\cap\S_{\beta_{n-1}}=\emptyset$, while it is obvious that $\S_{\beta_n}\cap\S_{\beta_i}=\emptyset$ for $i=0,\ldots,n-2.$ Thus we repeat this process by choosing $\beta=\beta_i\not\in\langle \pi \rangle\cup\S_{\beta_{n-1}}\cup\ldots\cup \S_{\beta_{i-1}},$ where $\S_{\beta_{i}}\cap \S_{\beta_{j}}=\emptyset$ for $j=0,\ldots,i-1$ and $i\geq (n-1);$ such that $ \langle \pi \rangle\bigcup_{i\geq(n-1)} \S_{\beta_i}$ is the set of all $n$-cycles in $S_n.$ 
Hence we have $k=\frac{(n-1)!-(n-1)}{n}$ non isomorphic digraphs $G_{\{\pi,\beta_i\}}$ whenever $\beta_i\not\in \langle \pi \rangle$ and $i=n-1,\ldots,n-2+k.$ Thus
$$\{G_{\{\pi,\beta\}}:\beta\in S_n\, \mbox{is a}\, n\mbox{-cycle}\}
=\bigcup_{\beta_i,i=0,\ldots,n-2+k}\{G_{\{\pi,\beta'_i\}}:\beta'_i\in \S_{\beta_i}\}, k=\frac{(n-1)!-(n-1)}{n}. $$ This completes the proof.
\end{proof}

\begin{proposition} \label{prop:1 n prime}
Let $n>2$ be a prime number and $\pi \in S_n$ be an $n$-cycle. Suppose $\beta \in S_n$ is not conjugate to $\pi,$ i.e. $ct(\beta)$ is other than the $(n).$ Then the number of equivalence graph classes in $\{G_{\{\pi',\beta'\}}: \pi', \beta'\in S_n\,\, \mbox{with}\,\, ct(\pi')=ct(\pi)\,\, \mbox{and}\,\, ct(\beta')=ct(\beta)\}$ is $N^{\beta}/n.$ 
\end{proposition}

\begin{proof} 
 Since $\beta$ is not an $n$-cycle, by Lemma \ref{lemma:n prime} $\beta \not \in \langle \pi \rangle.$ Then similarly as Proposition \ref{pro:both n prime} each $\S_{\beta}=\{f\beta f^{-1}:f\in \langle \pi \rangle\}$ contains $n$ distinct elements and $\S_{\beta_i}\cap \S_{\beta_j}=\emptyset$ for $\beta_i\neq \beta_j,$ where both $\beta_i$ and $\beta_j$ are conjugate to $\beta.$ Hence the number of non isomorphic digraphs in  $\{G_{\{\pi',\beta'\}}: \pi'\,\mbox{and}\, \beta'\, \mbox{are conjugate to}\, \pi\, \mbox{and}\,\beta\,\mbox{respectively}\}$ is equal to $\frac{1}{n}$ times the number of elements in $S_n$ having the same cycle type as $\beta.$
\end{proof}

Before the next proposition, we revisit the Goldbach's conjecture, which is verified computationally for numbers less than or equal to $4\times 10^{18}.$ The conjecture states that every even number greater than $2$ is the sum of two prime numbers \cite{wanggoldbach}.

\begin{proposition} \label{pro:two compo prime}
Let $n=k_1+k_2,$ where both $k_1,k_2\in \mathbb{N}$ be prime numbers.
Then, the number of equivalence graph classes in
 $\{G_{\{\pi,\beta\}}:\beta,\pi\in S_n\,\mbox{are of cycle type}\,(k_1,k_2)\}$ is $(k_1-1)(k_2-1)+(k_1-1)k+(k_2-1)l+m,$ where
 \begin{eqnarray*}
 && k=\frac{(k_2-1)!-(k_2-1)}{k_2},l=\frac{(k_1-1)!-(k_1-1)}{k_1} \\ && m=\frac{\frac{n!}{k_1k_2}-(k_1-1)(k_2-1)-k_2(k_1-1)k-k_1(k_2-1)l}{k_1k_2}.
 \end{eqnarray*}
\end{proposition}

\begin{proof}
 Let $\pi=\pi_1\pi_2 \in S_n$ be of cycle type $(k_1,k_2),$ where $\pi_1$  and $\pi_2$ are disjoint cycles of length $k_1$ and $k_2$, respectively, and $\S_{\beta}=\{f\beta f^{-1}:f\in \langle \pi_1,\pi_2 \rangle\}$ be a multi set. Let $\beta\in S_n$ be such that $ct(\pi)=ct(\beta).$ Clearly $N^{\pi}=\frac{n!}{k_1k_2}.$

First consider $\beta \in \langle \pi_1,\pi_2 \rangle,$ then by Lemma \ref{lemma:n=k_1k_2}, $f\beta f^{-1}=\beta, f \in \langle \pi_1,\pi_2 \rangle.$ So that $\S_{\beta}=\{\beta\}.$ Hence if $\beta,\beta'\in \langle \pi_1,\pi_2 \rangle \setminus \{id\}$ are of cycle type $(k_1, k_2)$ with $\beta\neq\beta',$ then
$\S_{\beta}\cap \S_{\beta'}=\emptyset.$
Thus by Remark \ref{remark:k_1k_2prime}, we have $(k_1-1)(k_2-1)$ non isomorphic digraphs $G_{\{\pi,\beta_{ij}\}},$ where $\beta_{ij}=\pi_1^{i}\pi_2^{j}$ for $i=1,\ldots,k_1-1$ and $j=1,\ldots,k_2-1.$

Now we choose $\beta \not\in \langle \pi_1,\pi_2 \rangle$ such that $\beta\in \langle \pi_1,\bar{\pi}_2 \rangle,$ where $\bar{\pi}_2\not \in \langle \pi_2\rangle$ is a $k_2$-cycle on the set $\{k_1+1,\ldots,n\}.$ Clearly $\beta$ is of the form $\pi_1^i\bar{\pi}_2^j$ for $1\leq i \leq k_1-1$ and $1\leq j \leq k_2-1.$ If $id\neq f\in \langle \pi_1,\pi_2 \rangle$ is of cycle type $(k_1,1,\ldots,1),$ i.e. $f=\pi_1^p$ for some $ p\in \{1,\ldots, k_1-1\},$ then $f\beta f^{-1}=f\pi_1^i\bar{\pi}_2^j f^{-1}=f\pi_1^if^{-1}f\bar{\pi}_2^j f^{-1}=f\pi_1^i f^{-1}\bar{\pi}_2^j.$ Hence by Lemma \ref{lemma:full cycle}, $f\beta f^{-1}=\pi_1^i\bar{\pi}_2^j=\beta.$
If $id\neq f\in \langle \pi_1,\pi_2 \rangle$ is of cycle type $(1,\ldots,1,k_2),$ i.e. $f=\pi_2^q$ for some $ q\in \{1,\ldots, k_2-1\},$ then $f\beta f^{-1}=f\pi_1^i\bar{\pi}_2^j f^{-1}=\pi_1^if\bar{\pi}_2^j f^{-1}.$ Since $k_2$ is prime then $\bar{\pi}_2\not \in \langle \pi_2\rangle$ yields $\bar{\pi}_2^j\not \in \langle \pi_2\rangle$ for all $j=1,\ldots,k_2-1.$ Now by Lemma \ref{lemma:full cycle}, $f\bar{\pi}_2^j f^{-1}=\bar{\pi}_2^j$ if and only if $\bar{\pi}_2^j\in \langle f \rangle =\langle \pi_2 \rangle,$ which is a contradiction. So, for all such $f,$  the elements in $\S_{\beta}$ are distinct, i.e. there are $k_2-1$ distinct elements. Further, if $id\neq f\in \langle \pi_1,\pi_2 \rangle$ is of cycle type $(k_1, k_2),$ i.e. $f=\pi_1^p\pi_2^q$ for some $p\in \{1,\ldots,k_1-1\}, q\in \{1,\ldots, k_2-1\},$ then $f\beta f^{-1}=f\pi_1^i\bar{\pi}_2^j f^{-1}=\pi_1^p\pi_1^i(\pi_1^p)^{-1} \pi_2^q\bar{\pi}_2^j (\pi_2^q)^{-1}=\pi_1^i \pi_2^q\bar{\pi}_2^j (\pi_2^q)^{-1}.$ So that this case turns into the case when $f$ is a $(1,\ldots,1,k_2)$ cycle, i.e. $f=\pi_2^q$ for some $ q\in \{1,\ldots, k_2-1\}.$ Thus all these result $1+k_2-1=k_2$ distinct elements in $\S_{\beta}.$ Therefore if we fix one $\pi_1,$ then by Proposition \ref{pro:both n prime}, there are $k$ choices for such $\beta,$ where $k$ is mentioned in the statement of the proposition. Hence for $\beta\in \langle \pi_1,\bar{\pi}_2 \rangle$ and $\bar{\pi}_2\not \in \langle \pi_2\rangle$ we have $(k_1-1)k$ non isomorphic digraphs $G_{\{\pi,\beta_{ij}\}}$ for $i=1,\ldots,k_1-1,j=k_2,\ldots,k_2-1+k.$

Similarly, for $\beta \not\in \langle \pi_1,\pi_2 \rangle$ and $\beta\in \langle \bar{\pi}_1,\pi_2 \rangle,$ where $\bar{\pi}_1\not \in \langle \pi_1\rangle$ is a $k_1$-cycle on the set $\{1,\ldots,k_1\},$ we get $k_1$ distinct elements in $\S_{\beta}$ and we have $l(k_2-1)$ non isomorphic digraphs $G_{\{\pi,\beta_{ij}\}}$ for $i=k_1,\ldots,k_1-1+l,j=1,\ldots,k_2-1,$ where $l$ is mentioned in the statement of the proposition.

Finally, consider $\beta \not\in \langle \pi_1,\pi_2 \rangle$ such that $\beta\in \langle \bar{\pi}_1,\bar{\pi}_2 \rangle,$ with $\bar{\pi}_1\not \in \langle \pi_1\rangle$ is a $k_1$-cycle and $\bar{\pi}_2\not \in \langle \pi_2\rangle$ is a $k_2$-cycle on the sets $\{1,\ldots,k_1\}$ and $\{k_1+1,\ldots,n\}$ respectively.
If $id\neq f \in \langle \pi_1,\pi_2 \rangle$ is of cycle type $(k_1,k_2),$ then by Lemma \ref{lemma:n=k_1k_2}, $f\beta f^{-1}=\beta$ indicates $\beta\in \langle \pi_1,\pi_2 \rangle,$ a contradiction.
If $id\neq f\in \langle \pi_1,\pi_2 \rangle$ is of cycle type $(k_1,1,\ldots,1)$ or $(1,\ldots,1,k_2)$ then clearly by Lemma \ref{lemma:k+1..+1} the relation $f\beta f^{-1}=\beta$ implies $\beta\in \langle \pi_1 \rangle\ast S_{n-k_1}$ or $\beta\in S_{n-k_2}\ast\langle \pi_2\rangle$\ respectively.  But $\beta$ is of cycle type $(k_1,k_2)$ and hence from both  the cases we obtain $\beta\in \langle \pi_1,\pi_2\rangle.$ This contradiction shows that there are $k_1k_2$ distinct elements in $\S_{\beta}.$ 
Thus whenever $\beta\in \langle \bar{\pi}_1,\bar{\pi}_2 \rangle$ with $\bar{\pi}_1\not \in \langle \pi_1\rangle$ and $\bar{\pi}_2\not \in \langle \pi_2\rangle,$ following a similar process as Proposition \ref{pro:both n prime}, we have $m$ non isomorphic digraphs, where $m$ is mentioned in the statement of the proposition.
\end{proof}

Then we have the following corollary whose proof is immediate from Propositions \ref{pro:both n prime}, \ref{prop:1 n prime}, and \ref{pro:two compo prime}.

\begin{corollary} \label{lemma:equi}
 The number of equivalence graph classes in
\begin{itemize}
    \item[(1)] $\{G_{\{\pi,\beta\}}:\beta,\pi\in S_5\,\mbox{are}\,5\mbox{-cycles}\}$ is $8,$ 
    \item[(2)] $\{G_{\{\pi,\beta\}}:\beta,\pi\in S_5,\,\mbox{where one of}\,\beta\,\mbox{or}\,\pi\,\mbox{is a}\,5\mbox{-cycle and the other is of cycle type}\,(3,2)\}$ is $4,$
    \item[(3)] $\{G_{\{\pi,\beta\}}:\beta,\pi\in S_5\,\mbox{are of cycle type}\,(3,2)\}$ is $5.$
    \end{itemize}
    \end{corollary}
Consequently, we have the following estimates for the trace of the product of permutation matrices with specific cycle types.
Observe that if $\pi=\pi_1\pi_2\in S_5$ is of cycle type $(3,2),$ then $\langle \sigma_1,\sigma_2\rangle=\langle \sigma \rangle.$ Now for $\pi,\beta,\beta' \in S_5,$ if $G_{\{\pi,\beta\}}$ and $G_{\{\pi,\beta'\}}$ are isomorphic digraphs, then $\beta'\in \S_{\beta},$ i.e. there exists $f\in \langle \pi \rangle$ such that $\beta'=f\beta f^{-1}.$ 
Hence $tr(P_{\beta'}P_{\pi})=tr(P_f^TP_{\beta}P_fP_f^TP_{\pi}P_f)=tr(P_{\beta}P_{\pi}).$ Using this fact, we have the following estimates for the trace of the product of permutation matrices with the specific cycle types.


 \begin{corollary} \label{cor:tr}
Let $P_{\beta},P_{\pi}\in \P_5$ where $\beta,\pi\in S_5.$ Then the following are true.
\begin{itemize}
    \item[(1)] $tr(P_{\beta}P_{\pi}) \in \{0,1,2,5\}$ if both $\beta,\pi$  are $5$-cycles.
    \item[(2)] $tr(P_{\beta}P_{\pi}) \in \{0,1,3\}$ if one of  $\beta,\pi$ is a $5$-cycle and the other is of cycle type $(3,2).$
    \item[(3)] $tr(P_{\beta}P_{\pi}) \in \{0,1,2,5\}$ if both $\beta,\pi$  are of cycle type $(3,2).$  
\end{itemize}
    \end{corollary}

Even if the proofs of Corollary \ref{lemma:equi} and Corollary \ref{cor:tr} are clear from the above propositions, we mention in Table \ref{tab:table1} the explicit representations of permutations that belong to graph equivalence classes for different sets of digraphs $G_{\{\pi,\beta\}}$ as mentioned in Corollary \ref{lemma:equi}. 

\begin{table}[h!]
  \begin{center}
    \begin{tabular}{|l|c|} 
      \hline
      $G_{\{\pi,\beta\}} $ & $[G_{\{\alpha,\beta\}}], \alpha, \beta \in S_5$   \\
      $\pi, \beta \in S_5,$   $a_i\in\{1,\hdots,5\}$ & $\{\alpha,\beta\}$ \\
      \hline \hline
      $\pi, \beta$ are $5$-cycles  & $\{\pi, \pi\}, \{\pi, \pi^2\}, \{\pi, \pi^3\}, \{\pi, \pi^4\}$\\
    $ \pi=(a_1a_2a_3a_4a_5)$ & $\{\pi, (a_1a_2a_3a_5a_4)\}, \{\pi, (a_1a_2a_4a_5a_3)\}, $ \\
     & $\{\pi, (a_1a_2a_5a_4a_3)\}, \{\pi, (a_1a_3a_5a_4a_2)\}$\\
      \hline
      one of $\pi, \beta$ is $5$-cycle and & $\{\pi,(a_1a_2a_3)(a_4a_5)\}, \{\pi,(a_1a_2a_4)(a_3a_5)\}$ \\
      another has of cycle type $(3,2)$ & $\{\pi,(a_1a_3a_2)(a_4a_5)\}, \{\pi,(a_1a_4a_2)(a_3a_5)\}$ \\ 
      $ \pi=(a_1a_2a_3a_4a_5)$ & \\
      \hline
      $\pi,\beta$ have cycle type $(3,2)$ & $\{\pi,\pi\}, \{\pi,\pi^5\}$ \\
       $ \pi=(a_1a_2a_3)(a_4a_5)$& $\{\pi,(a_1a_2a_4)(a_3a_5)\}$, $\{\pi,(a_1a_4a_2)(a_3a_5)\}$\\
       & $\{\pi, (a_1a_4a_5)(a_2a_3)\}$ \\
       \hline
    \end{tabular}
    \caption{Equivalence graph classes}   \label{tab:table1}
  \end{center}
\end{table}


\section{$k$-th power of trace-zero DS matrices} \label{sec:3}

In this section, we prove that the square of a matrix in $\Omega_5^0$ cannot be the identity matrix.
Then we determine the trace-zero DS matrices whose $k$-th power is also a trace-zero matrix, where $k\in\{2,3,4,5\}.$  
We show that if the trace of the square or the cube of a matrix in $\Omega_5^0$ is zero then the matrix is a convex combination of at most three or two permutation matrices, respectively. Finally, we prove that all such $A\in \Omega_5^0,$ for which either $\mathrm{tr}(A^4)=0$ or $\mathrm{tr}(A^5)=0$ are permutation matrices corresponding to $5$-cycle or $(3,2)$ cycle, respectively.




\begin{proposition} \label{pro_tr5}
Let $A \in \Omega_5^0.$ Then $tr(A^2)\neq5.$
\end{proposition}

\begin{proof}  Let $tr(A^2)=5$ and hence $A^2=I.$ Then $A=(a_{ij})\notin \P_5^0.$ Let $A$ be a convex combination of at least two distinct permutation matrices, i.e. there is at least one row and one column of $A$ with more than one nonzero entry. Suppose $A^2=(a_{ij}^{(2)}).$ Then for $i,j=1,\ldots,5$ and $i \neq j,$ $a_{ij}^{(2)}=\sum_{k=1}^5 {a_{ik} a_{kj}}=0$ yield either $a_{ik}=0$ or $a_{kj}=0$ for $k=1,\ldots,5.$ Suppose $a_{ik_1}, \ldots, a_{ik_m}$ are the nonzero entries in the $i^{th}$ row of $A,$ where $2\leq k_m\leq 4$ and $k_j\neq i$ for $j=1,\ldots,m.$  Then $a_{k_1j}=\ldots=a_{k_mj}=0$ for all $j \neq i.$ Now, since $A$ is a DS matrix, we have for $k \in \{k_1,\ldots, k_m\},$ $\sum_{j=1}^5{a_{kj}}=a_{ki}=1,$  which implies that the $i^{th}$ column of $A$ contains $k$ nonzero elements equal to $1,$ which is a contradiction. This completes the proof. 
\end{proof}


\begin{remark}
The above theorem can easily be extended for matrices $A \in \Omega_n^0$ when $n$ is odd. Indeed, note that if $n$ is odd then $P^2 \neq I,$ for any $P\in \P_n^0.$ Hence $tr(A^2) \neq n$ can be proved using similar arguments used in the above proof.  
\end{remark}


\begin{proposition} \label{prop:A^2}
Let $A \in \Omega_5^0$ be such that $tr(A^2)=0.$ Then $A$ is a convex combination of at most three permutation matrices corresponding to $5$-cycles.
\end{proposition}

\begin{proof} Let $A=\sum_{\alpha}{c_{\alpha}P_{\alpha}},P_{\alpha}\in\P_5^0,$ $c_{\alpha}>0$ and $\sum_{\alpha}c_{\alpha}=1.$ So that $A^2=\sum_{\alpha}\sum_{\beta}c_{\beta}{c_{\alpha}P_{\beta}P_{\alpha}}.$ Now $tr(A^2)=0$ implies 
\begin{equation}\label{tr_0}
\sum_{\alpha=\beta}{c^2_{\alpha}P^2_{\alpha}}+\sum_{\alpha\neq\beta}{c_{\alpha}c_{\beta}P_{\alpha}P_{\beta}}=0.
\end{equation}
Since $c_{\alpha}c_{\beta}>0,$ $(\ref{tr_0})$ implies $tr(P_{\beta}P_{\alpha})=0$ and $tr(P_{\alpha^2})=0~\forall\,\, \alpha, \beta.$ 
If for some $\alpha,$ $P_{\alpha}$ corresponds to $(3,2)$ cycles then $tr(P_{\alpha}^2)=2$ and hence no matrix from $\{P_{\alpha}\}$ can correspond to cycle type $(3,2).$  Thus $A$ is a convex combination of permutation matrices that are $5$-cycles only. Clearly, if $A\in\P_5^0$ corresponds to a $5$-cycle permutation then $tr(A^2)=0.$

If possible let $A$ be the convex combination of more than two distinct permutation matrices corresponding to $5$-cycles, and suppose one of them is $P_{\pi}$ where ${\pi}=(a_1a_2a_3a_4a_5), a_i\in \{1,\hdots,5\}.$ Then, from Table \ref{tab:table1}, there are  ${\beta_1}:=\pi^2=(a_1a_3a_5a_2a_4),\beta_2:=\pi^3=(a_1a_4a_2a_5a_3)$ and ${\beta_5}:=(a_1a_2a_4a_5a_3)$ such that  $P_{\pi\beta_1}, P_{\pi\beta_2}$ and $P_{\pi\beta_5}$ have zero traces.
Now $\S_{\beta_5} = \{\beta_5,(a_1a_2a_5a_3a_4), (a_1a_3a_4a_2a_5),$ $ (a_1a_4a_2a_3a_5), (a_1a_4a_5a_2a_3)\},$ $\S_{\beta_1}=\{{\beta_1}\}, \S_{\beta_2}=\{{\beta_2}\}.$ Then $tr(P_{\beta_5}P_{\pi})=tr(P_{\beta}P_{\pi}),tr(P_{\beta_5}P_{\beta_1})=tr(P_{\beta}P_{\beta_1})$ and $tr(P_{\beta_5}P_{\beta_2})=tr(P_{\beta}P_{\beta_2}),\,\forall\, \beta\in \S_{\beta_5}$ so that
\begin{equation}\label{A:2}
A=\sum_{\alpha\in\{\pi\}\cup\S_{\beta_1}\cup\S_{\beta_2}\cup\S_{\beta_5}}{c_{\alpha}P_{\alpha}}.
\end{equation}
 
 However, since $tr(P_{\beta_i\beta_j})\neq 0,$ for $i\neq j, i,j=1,2,5,$ (\ref{A:2}) contradicts $tr(A^2)= 0.$ Thus 
$A$ takes the form $\sum_{\alpha\in\{\pi\}\cup\S_{\beta_i}}{c_{\alpha}P_{\alpha}}$ for $i=1,2,5.$  Now it is easy to verify that we can take maximum two elements from $\S_{\beta_5},$ such that the permutation matrix corresponding to their product has zero trace. Thus, $A$ is the convex combination of at most three permutation matrices.
Hence the proof follows.
\end{proof}


\begin{proposition} \label{prop:A^3}
Let $A \in \Omega_5^0$ be such that $tr(A^3)=0.$ Then $A$ is a convex combination of at most two permutation matrices corresponding to $5$-cycles.
\end{proposition}

\begin{proof} Let $A=\sum_{\alpha}{c_{\alpha}P_{\alpha}}$ for $P_{\alpha}\in\P_5^0,$ with $c_{\alpha}>0,\sum_{\alpha}c_{\alpha}=1.$ 
Now, $$A^3=\sum_{\alpha=\beta=\gamma}{c^3_{\alpha}P^3_{\alpha}}+\sum_{\alpha\neq\beta\neq\gamma}{c_{\alpha}c_{\beta}c_{\gamma}P_{\alpha}P_{\beta}P_{\gamma}}.$$ Thus $P^3_{\alpha},P_{\alpha}P_{\beta}P_{\gamma}\in \P_5^0$ whenever $tr(A^3)=0,$ and hence 
each one from the collection $\{P_{\alpha}\}$ corresponds to a $5$-cycle only. 


Let $A$ be the convex combination of more than one permutation matrix, and we write \begin{equation}A^3=\sum_{\gamma}{c_{\gamma}P_{\gamma}}\left(\sum_{\alpha=\beta}{c^2_{\alpha}P^2_{\alpha}}+\sum_{\alpha\neq\beta}{c_{\alpha}c_{\beta}P_{\alpha}P_{\beta}}\right)
\end{equation} such that for $tr(A^3)=0$ all the $\alpha^2\gamma$ and $\alpha\beta\gamma$ are $5$-cycles when $\alpha,\beta,\gamma$ are $5$-cycles.  
 Then, from Table \ref{tab:table1}, if $P_{\pi}\in \{P_{\alpha}\}$ such that
$\pi^2=(a_1a_2a_3a_4a_5),$ then there are ${\beta_1} :=\pi^4=(a_1a_3a_5a_2a_4), \beta_2 :=\pi=(a_1a_4a_2a_5a_3)$ and ${\beta_5} :=(a_1a_2a_4a_5a_3)$
such that $P_{\pi^2\beta_1},P_{\pi^2\beta_2}$ and $P_{\pi^2\beta_5}$ have zero traces. 
Then,  similar to the proof of Proposition \ref{prop:A^2}, consider $A=\sum_{\alpha\in \{\pi\}\cup\S_{\beta_1}\cup\S_{\beta_5}}{c_{\alpha}P_{\alpha}}$ since $\beta_2=\pi.$
Now clearly $P_{\beta^2_i\beta_j}$ for $i\neq j,i,j\in \{1,5\}$ are present in the expression of $A^3,$ where $\beta_1\neq \beta_5.$ But  
 $tr(P_{\beta^2_5\beta_1})=tr(P_{\beta_5\beta^2_1})=2$  
 whereas $tr(P_{\beta^2_5\pi})=tr(P_{\beta^2_1\pi})=0.$
 Since $tr(P_{\beta_5^2}P_{\beta_1})=tr(P_{\beta^2}P_{\beta_1})$ for all $\beta \in \S_{\beta_5},$ the permutation matrices correspond to the permutations from the sets $\S_{\beta_5}$ and $\S_{\beta_1}=\{\beta_1\}$ cannot be present simultaneously in $A.$ 
Now it is easy to verify that for any two elements $\pi_1,\pi_2 \in \S_{\beta_5},$ $tr(P_{\pi_1^2\pi_2})\neq 0.$

Thus in the expression of $A$ we can choose only one permutation matrix which belongs to $\S_{\beta_5}.$ Thus either $A=\sum_{\alpha\in \{\pi,\beta_1\}}{c_{\alpha}P_{\alpha}}$ or $A=\sum_{\alpha\in \{\pi,\beta\}}{c_{\alpha}P_{\alpha}}$ for $\beta \in \S_{\beta_5}.$
Hence the proof follows.
\end{proof}

\begin{proposition} \label{prop:A^4}
Let $A \in \Omega_5^0$ be such that $tr(A^4)=0.$ Then $A$ is a $5$-cycle permutation matrix.
\end{proposition}

\begin{proof}Let $A=\sum_{\alpha}{c_{\alpha}P_{\alpha}}$ for $P_{\alpha}\in\P_5^0,$ with $c_{\alpha}>0,\sum_{\alpha}c_{\alpha}=1.$ 
Then clearly $P^4_{\alpha}$ is present in the expression of $A^4$ and hence for $tr(A^4)=0,$ each one from the collection $\{P_{\alpha}\}$ corresponds to a $5$-cycle only. 
Similar to the proof of the last two propositions, if $P_{\pi}\in \{P_{\alpha}\}$ such that
$\pi^3=(a_1a_2a_3a_4a_5),$ then there are ${\beta_1} :=\pi=(a_1a_3a_5a_2a_4), \beta_2 :=\pi^4=(a_1a_4a_2a_5a_3)$ and ${\beta_5} :=(a_1a_2a_4a_5a_3)$
such that $P_{\pi^3\beta_1},P_{\pi^3\beta_2}$ and $P_{\pi^3\beta_5}$ have zero traces. 
Then we write $A=\sum_{\alpha\in \{\pi\}\cup\S_{\beta_2}\cup\S_{\beta_5}}{c_{\alpha}P_{\alpha}}.$
Now $tr(P_{\beta^3_5\beta_2})=tr(P_{\beta_5\beta^3_2})=2$  
 whereas $tr(P_{\beta^3_2\pi})=tr(P_{\beta^3_5\pi})=0.$
 Thus, the permutation matrices corresponding to the permutations from the sets $\S_{\beta_5}$ and $\S_{\beta_2}$ cannot be present simultaneously in $A.$ 
Also, $P_{\gamma^2}P_{\beta^2}$ is present in $A^4$ for $P_{\gamma},P_{\beta}\in \{P_{\alpha}\},$ but neither $tr(P_{\pi^2\beta_2^2})$ nor $tr(P_{\pi^2\beta_5^2})$ are equal to zero.

Thus, we conclude that $A\in \P_5^0$, which corresponds to a $5$-cycle. Hence the proof follows.
\end{proof}

\begin{proposition} \label{prop:A^5}
Let $A \in \Omega_5^0$ be such that $tr(A^5)=0.$ Then $A\in \P_5^0$ and corresponds to a permutation whose cycle type is $(3,2).$  
\end{proposition}

\begin{proof}
Since $tr(A^5)=0,$ $A$ can be the convex combination of permutation matrices all corresponding to cycle type $(3,2)$ only. 
Let $A=\sum_{\alpha}{c_{\alpha}P_{\alpha}}$ for $P_{\alpha}\in\P_5^0.$ Then for any two permutation matrices $P_{\pi},P_{\beta}\in \{P_{\alpha}\},$ the permutation matrix $P_{\pi^3\beta^2}$ should appear in the expression of $A^5$ and hence $tr(P_{\pi^3\beta^2})=0.$ 
Now since both $\pi$ and $\beta$ have cycle type $(3,2),$ 
 it is easy to verify that $\beta$ equals $\pi$ for $tr(P_{\pi^3\beta^2})=0.$ Thus, $A$ cannot be written as the nontrivial convex combination of permutation matrices. Which completes the proof.
\end{proof}

\section{Characteristic polynomial and eigenvalue region of trace-zero DS matrices of order $5$ 
}\label{sec:4}

In this section, 
we derive certain necessary conditions for the coefficients of a degree $5$ generic polynomial to be the characteristic polynomial of a trace zero DS matrix of order $5$. We verify these conditions with numerous numerical examples. Finally, we provide an approximation for the region $\omega^0_5.$   

In the following theorem, we provide inequality bounds for the coefficients of the characteristic polynomial of a trace zero DS matrix of order $5$ using some of the results mentioned in Section \ref{sec:preli}. 
Note that any trace zero DS matrix can be written as a convex combination of trace zero permutation matrices. Accordingly, pairs of trace zero permutation matrices can decide the probable boundaries of $\omega_5^0$. Consequently, we consider the convex combinations $A=cP_{\pi}+(1-c)P_{\beta}, c<0<1$ corresponding to permutation matrix pairs $\{P_{\pi}, P_{\beta}\}$ from the column on the right hand side of Table \ref{tab:table1}. Then we choose those matrices $A$ for which the coefficients of the characteristic polynomials have maximum or minimum values. Thus, adapting a similar {approach to that} of the trace zero DS matrices of order $4$ given in \cite{Benvenuti, Benvenuti2018}, we derive relations among the coefficients by eliminating the parameter $c$ and establish the following theorem. 

\begin{theorem} \label{theorem:ineq}
Let $P_A(x)=x^5+k_1x^4+k_2x^3+k_3x^2+k_4x+k_5$ be the characteristic polynomial of a matrix $A \in \Omega_5^0.$ Then



\begin{enumerate}
\item[(a)] $k_1=0$,
\item[(b)] $-\frac{5}{2} < k_2 \leq 0$ and $-\frac{5}{3} \leq k_3 \leq 0,$
\item[(c)] $k_3$ and $k_4$ satisfy at least one of the following conditions. \begin{enumerate}[(i)] 
\item $\begin{cases}{\max}\{k_2,-2-k_2\}\leq k_3\leq 0\,\,,\\
       \begin{cases}
       k_4 \leq 1+k_3 \,\,\,\mbox{if}\,\, k_4 \geq 0\\
      k_4 \geq {\max}\{-(1+k_2)(4+5k_2),-\frac{1}{9}(1+k_3)(2+5k_3)\}
       \,\,\,\mbox{if}\,\, k_4 \leq 0;
      \end{cases}
      \end{cases}$
   \item  $\begin{cases} \label{thm:ieq2}
   -\frac{5}{4}\leq k_2\leq 0\,\,,\\
     -\frac{20}{27}\leq k_3\leq 0\,\,,\\
     \begin{cases}
     k_4 \leq \frac{k_2^2}{5}\,\,\,\mbox{if}\,\, k_4 \geq 0\\
      \min\{k_3, k_4\} \geq -5 c_*(1-c_*)^2\,\,\,\mbox{if}\,\, k_4<0,\,k_3\neq0\,\,\mbox{and}\,\,\,0<\frac{k_4}{k_3}<\frac{1}{6}\\
       k_4 \geq -5 c_*(1-c_*)^2\,\,\,\mbox{if}\,\, k_4<0,\,k_3\neq0\,\,\mbox{and}\,\,\,\frac{1}{6}\leq \frac{k_4}{k_3}<1\\
       k_3 \geq -5 c_*^3(1-c_*)\,\,\,\mbox{if}\,\, k_4<0,\,k_3\neq0\,\,\mbox{and}\,\,\,1<\frac{k_4}{k_3}\leq \frac{9}{4}\\
       \min\{k_3, k_4\} \geq -5 c_*^3(1-c_*)\,\,\,\mbox{if}\,\, k_4<0,\,k_3\neq0\,\,\mbox{and}\,\,\,\frac{9}{4}<\frac{k_4}{k_3},\\ \hfill{\mbox{where}\,\,c_*=\frac{1}{2}\left(-\frac{k_4}{k_3}+\sqrt{\left(\frac{k_4}{k_3}\right)^2+4\frac{k_4}{k_3}}\right)}\\
         k_4 = 2 k_2\,\,\,\mbox{if}\,\, k_4 \leq 0\,\,\mbox{and}\,\,k_3=0;
      \end{cases}
      \end{cases}$
    \item  $\begin{cases}-(-k_3)^{\frac{1}{3}}(3-2(-k_3)^{\frac{1}{3}})\leq k_2\,\,,\\
       \begin{cases}
       k_4\geq-3 \sqrt{-k_2}(1-\sqrt{-k_2})^2
 \,\,\,\mbox{if}\,\, k_4 \leq 0\\
       k_4\leq-(-k_3)^{\frac{1}{3}}\left(1-(-k_3)^{\frac{1}{3}}\right)\left(1-3(-k_3)^{\frac{1}{3}}\right) \,\,\,\mbox{if}\,\, k_4 \geq 0;
      \end{cases}
      \end{cases}$
      \end{enumerate}
      \item[(d)] $1+k_1+k_2+k_3+k_4+k_5=0.$
\end{enumerate}
\end{theorem}
\pf
Newton's identity \cite{kalman2000matrix} for the coefficients of characteristic polynomial in terms of the trace of a matrix is given by
$$
k_i=-\frac{1}{i}[k_{i-1} tr(A)+k_{i-2} tr(A^2)+\ldots+k_0 tr(A^i)], k_0=1.
$$
Then for our case, $k_1=0,k_2=-\frac{1}{2}tr(A^2)\in (-\frac{5}{2},0]$ by Proposition \ref{pro_tr5}, and $k_3=-\frac{1}{3}tr(A^3)\in [-\frac{5}{3},0].$ Thus the proofs of $(a)$ and $(b)$ follow. Besides, we get
\begin{align} \label{coeffi:k4}
\begin{split}
k_4=&-\frac{1}{4}[k_3 tr(A)+k_2 tr(A^2)+k_1 tr(A^3)+k_0 tr(A^4)] \\
    =&-\frac{1}{4}[-2 k_2^2 + tr(A^4)].
    \end{split}
    \end{align}
    Since $1$ is an eigenvalue of the doubly stochastic matrix $A,$ $P_A(1)=0$ and hence leads the proof of $(d)$.

Clearly from (\ref{coeffi:k4}), unlike $k_2$ and $k_3,$ $k_4$ is not specified to be non negative or non positive. The upper and lower bound for the coefficient $k_4$ can be obtained while we consider $A=(1-c)Q+cR,\,Q,R\in \P_5^0,Q\neq R$ and $0<c<1$ \cite{Benvenuti2018}. Then we determine the coefficients $k_i,i=2,3,4$ below.

\begin{enumerate}[(i)]
\item
Suppose $R= P_{\pi}$ and $Q= P_{\beta},$ where both $\pi$ and $\beta$  are $(3,2)$ cycles and $\pi\neq \beta$. Hence $tr(R^4)=tr(Q^4)=2,tr(R^3)=tr(Q^3)=3.$ Then from Table \ref{tab:table1} the following possible cases arise.
\begin{itemize}
\item[(A)] \underline{$\pi=(a_1a_2a_3)(a_4a_5),\beta=(a_1a_3a_2)(a_4a_5).$}\\
As $R=Q^{T},$ $tr(QR)=tr(Q^2R^2)=tr(QRQR)=5,tr(Q^2R)=tr(QR^2)=0,tr(Q^3R)=tr(QR^3)=2.$ So that $k_2=-\left(1+3 c(1-c)\right),k_3=-\left(1-3c(1-c)\right)$ and $ k_4=3c(1-c) > 0 ~\mathrm{since}~0< c< 1$.

\item[(B)] \underline{$\pi=(a_1a_2a_3)(a_4a_5),$ $\beta=(a_1a_2a_4)(a_3a_5).$}\\
For this case $tr(QR)=tr(Q^3R)=tr(QR^3)=tr(QRQR)=0,tr(Q^2R)=tr(QR^2)=tr(Q^2R^2)=1.$ Hence $k_2=k_3=-(c^2+(1-c)^2)$, and $k_4=0.$

\item[(C)] \underline{$\pi=(a_1a_2a_3)(a_4a_5),$ $\beta=(a_1a_4a_2)(a_3a_5).$}\\
Here $tr(QR)=1,tr(Q^3R)=tr(QR^3)=tr(Q^2R)=tr(QR^2)=0,tr(Q^2R^2)=2,tr(QRQR)=5.$ So that, $k_2=c(1-c)-1,k_3=3 c(1-c)-1$ and $k_4=c(1-c)(1-5c(1-c)).$

\item[(D)] \underline{$\pi=(a_1a_2a_3)(a_4a_5),$ $\beta=(a_1a_4a_5)(a_2a_3).$}\\
$tr(QR)=tr(QRQR)=2,tr(Q^3R)=tr(QR^3)=tr(Q^2R)=tr(QR^2)=1,tr(Q^2R^2)=0.$ Hence $k_2=-1,k_3=-(c^2+(1-c)^2),k_4=c(1-c).$
\end{itemize}
We get the maximum value of $k_4$ and the minimum value of $k_2$ in case $(A).$ So the maximum of $k_4$ is $k_4^{max}=3c(1-c)=1+k_3,$ and the minimum of $k_2$ is $k_2^{min}=-(1+3c(1-c))=-2-k_3.$ 
Similarly, for case $(C)$ we have the minimum of $k_4,$ $k_4^{min}= c(1-c)(1-5c(1-c))=-(1+k_2)(4+5k_2)=-\frac{1}{9}(1+k_3)(2+5k_3).$ Again, considering $(B)$ and $(D)$ for the minimum of $k_3$ we get $k_3^{min}=k_2$ or $k_2\leq k_3^{min}.$ Thus all these yield the inequalities in $(c)(i).$

\item Let $R=P_{\pi},Q= P_{\beta}$ and $\pi,\beta$ be $5$-cycles and $\pi \neq \beta$. Hence $tr(R^4)=tr(Q^4)=tr(R^3)=tr(Q^3)=0.$ Now from Table \ref{tab:table1} we have the following possible cases.
\begin{itemize}
\item[(A)] \underline{$\pi=(a_1a_2a_3a_4a_5),\beta=(a_1a_3a_5a_2a_4).$}\\
In this case $tr(QR)=tr(QR^2)=tr(Q^3R)=tr(Q^2R^2)=tr(QRQR)=0,tr(Q^2R)=tr(QR^3)=5.$ Hence $k_2=0,k_3=-5c(1-c)^2$ and $k_4=-5c^3(1-c).$
\item[(B)] \underline{$\pi=(a_1a_2a_3a_4a_5),\beta=(a_1a_4a_2a_5a_3).$}\\
Here $tr(QR)=tr(QR^3)=tr(Q^2R^2)=tr(Q^2R)=tr(QRQR)=0,tr(QR^2)=tr(Q^3R)=5.$ So that $k_2=0,k_3=-5c^2(1-c)$ and $k_4=-5c(1-c)^3.$
\item[(C)] \underline{$\pi=(a_1a_2a_3a_4a_5),\beta=(a_1a_5a_4a_3a_2).$}\\
Here $R=Q^T,$ so $tr(QR)=tr(Q^2R^2)=tr(QRQR)=5,tr(QR^2)=tr(Q^2R)=tr(QR^3)=tr(Q^3R)=0,tr(Q^4)=tr(R^4)=tr(Q^3)=tr(R^3)=0.$ Hence $k_2=-5c(1-c),k_3=0$ and $k_4=5c^2(1-c)^2.$
\item[(D)] \underline{$\pi=(a_1a_2a_3a_4a_5),\beta=(a_1a_2a_3a_5a_4).$}\\
In this case $tr(QR)=1,tr(Q^2R)=tr(QR^2)=0,tr(Q^3R)=tr(Q
R^3)=tr(Q^2R^2)=2,tr(QRQR)=5.$ Hence $k_2=-c(1-c),k_3=0$ and $k_4=-2 c(1-c).$

\item[(E)] \underline{$\pi=(a_1a_2a_3a_4a_5),\beta=(a_1a_2a_4a_5a_3).$}\\
Here $tr(QR)=tr(QRQR)=0,tr(Q^2R)=tr(QR^2)=tr(Q^2R^2)=2,tr(Q^3R)=tr(QR^3)=1.$ So that, $k_2=0,k_3=-2 c(1-c)$, and $k_4=-c(1-c).$

\item[(F)] \underline{$\pi=(a_1a_2a_3a_4a_5),\beta=(a_1a_2a_5a_4a_3).$}\\
In this case $tr(QR)=tr(Q^2R)=tr(QR^2)=tr(QRQR)=2,tr(Q^2R^2)=1,tr(Q^3R)=tr(QR^3)=0.$ So that $k_2=k_3=-2 c(1-c),$ and $k_4=0.$

\item[(G)] \underline{$\pi=(a_1a_2a_3a_4a_5),\beta=(a_1a_3a_5a_4a_2).$}\\
Here $tr(QR)=tr(Q^3R)=tr(QR^3)=tr(QRQR)=2,tr(Q^2R)=tr(QR^2)=1,tr(Q^2R^2)=0.$ So that $k_2=-2c(1-c),k_3=-c(1-c)$ and $k_4=-c(1-c)(5c^2-5c+2).$
\end{itemize}
We have $k_4^{max}=5c^2(1-c)^2=\frac{k_2^2}{5}$ in case $(C).$  
For $k_4<0$ and $k_3=0,$ we get from $(D)$ that $k_4=-{2}c(1-c)$ and the corresponding $k_2=-c(1-c),$ then by eliminating $c,$ and as well as from Proposition \ref{prop:A^3}, we get the desired result. Now $k_4^{min}$ and $k_3^{min}$ are obtained from $(A)$ and $(B),$ which have the same expressions while replacing $c$ with $(1-c).$ These give the same results after eliminating $c.$ Hence we go for one of them by taking $k_4^{min}= -5c^3(1-c)$ and $k_3^{min} =-5c(1-c)^2\neq 0$ where $0<c<1.$ Thus we have the desired inequalities given in $c(ii)$ for $k_4<0$ and $k_3\neq 0$ (see Appendix for details). Again, from $(A)$ we get $k_3^{min}=-5 c(1-c)^2$ with the corresponding $k_2=0$ and from $(C)$ we have $k_2^{min}=-5c(1-c)$ with the corresponding $k_3=0;$ hence $k_3^{min}\geq-\frac{20}{27}$ and $k_2^{min}\geq -\frac{5}{4}.$
\item
Consider the case when $R= P_{\pi}$ and $Q= P_{\beta}$ with the $5$-cycle $\pi$ and the $(3,2)$ cycle $\beta.$ Hence $tr(R^4)=tr(R^3)=0,tr(Q^3)=3,tr(Q^4)=2.$ Then by Table \ref{tab:table1} we consider the following possible cases.
\begin{enumerate}[(A)]
\item \underline{$\pi=(a_1a_2a_3a_4a_5),\beta=(a_1a_2a_3)(a_4a_5).$}\\
 $tr(QR)=tr(QR^2)=tr(Q^3R)=tr(QRQR)=1,tr(Q^2R^2)=tr(QR^3)=0,,tr(Q^2R)=2.$ So that $k_2=k_3=-(1-c)$ and $ k_4=0.$

\item \underline{$\pi=(a_1a_2a_3a_4a_5),$ $\beta=(a_1a_2a_4)(a_3a_5).$}\\
In can be checked that $tr(QR)=tr(Q^3R)=0,tr(QR^3)=3,tr(QR^2)=tr(Q^2R)=1,tr(R^2Q^2)=tr(QRQR)=2.$ Hence $k_2=-(1-c)^2,k_3=-(1-c)\left(1-c(1-c)\right)$, and $k_4=-3c^2(1-c).$

\item \underline{$\pi=(a_1a_2a_3a_4a_5),$ $\beta=(a_1a_3a_2)(a_4a_5).$}\\
Here $tr(QR)=3,tr(R^3Q)=tr(RQ^3)=tr(R^2Q^2)=1,tr(R^2Q)=tr(RQ^2)=0,tr(QRQR)=5.$ So that, $k_2=-(1-c)(1+2c),k_3=-(1-c)^3$ and $k_4=-c(1- c)(3c-2).$

\item \underline{$\pi=(a_1a_2a_3a_4a_5),$ $\beta=(a_1a_4a_2)(a_3a_5).$}\\
$tr(RQ)=tr(R^3Q)=tr(QRQR)=1,tr(RQ^3)=tr(R^2Q^2)=tr(RQ^2)=0,tr(R^2Q)=3.$ Hence $k_2=-(1-c),k_3=-(1-c)(3c^2+(1-c)^2),k_4=-c(1-c)(2c-1).$
\end{enumerate}
We obtain $k_4^{max}$ and $k_4^{min}$ in the cases $(C)$ and $(B),$ respectively. Consider case $(C),$ then $k_4^{max}=-c(1-c)(3c-2)$ and $k_2^{min}=-(1-c)(1+2c),$ with $k_3=-(1-c)^3.$ By substituting $c=1-(-k_3)^{\frac{1}{3}}$ in $k_4^{max},$ 
we get the desired relation. Similarly, we have $k_4^{min}=-3c^2(1-c)=-3(1-\sqrt{-k_2})^2\sqrt{-k_2}$ 
in case $(B),$  with $k_2=-(1-c)^2.$
Thus, all these results establish the inequalities present in $c(iii).$ \end{enumerate} 
These complete the proof. $\hfill{\square}$ 

Now we consider some numerical examples.

\begin{example}
\begin{enumerate}
    \item Consider $$\bmatrix{0&0.5&0&0.3&0.2\\0.1&0&0.7&0.2&0\\0.7&0&0&0&0.3\\0&0.3&0.2&0&0.5\\0.2&0.2&0.1&0.5&0}\in \Omega_5^0,$$ with the characteristic polynomial $x^5-0.43 x^3-0.436 x^2-0.1585 x+0.0245.$ So that, $-\frac{20}{27}<k_3, -\frac{5}{4}<k_2, k_3\leq k_2, k_4<0,$ and $1>\frac{k_4}{k_3}>\frac{1}{6}.$ Now, $k_4>-5c_*(1-c_*)^2,$ where $c_*\approx0.448.$ Thus, the relations given in $c(ii)$ of Theorem \ref{theorem:ineq} are satisfied.
    \item Take $$\bmatrix{0&0.1&0.2&0.3&0.4\\0.4&0&0.1&0.2&0.3\\0.4&0&0&0.4&0.2\\0&0.2&0.7&0&0.1\\0.2&0.7&0&0.1&0}\in\Omega_5^0.$$ Then for the characteristic polynomial $x^5-0.74 x^3-0.3 x^2+0.02x+0.02,$ we verify  $-\frac{20}{27}<k_3, -\frac{5}{4}<k_2, k_2<k_3$ and $k_4 >0.$ Further, $k_4=0.02 <0.10952=\frac{k_2^2}{5}$ holds as given in $c(ii)$ of Theorem \ref{theorem:ineq}.
\end{enumerate}
\end{example}

We verify the inequalities given in Theorem \ref{theorem:ineq} for a large number of examples of trace zero DS 
 matrix with the help of MATLAB. There are in total $45$ permutation matrices of order $5$ having zero trace, say, $P_{\pi_1},\ldots,P_{\pi_{45}}$ for $\pi_1,\ldots,\pi_{45}\in S_5^0.$ Since any matrix in $\Omega_5^0$ can be written as the convex combination of $P_{\pi_1},\ldots,P_{\pi_{45}}$ for $\pi_1,\ldots,\pi_{45}\in S_5^0,$  we construct the desired matrices of the form $\sum_{i=1}^{45}{c_{i}P_{\pi_i}}$ with $c_{i}\geq 0,\sum_{i=1}^{45}c_{i}=1.$ 
 Then in different runs, by using nested loops, we update the values of $44$ nonzero coefficients $ c_1,\ldots,c_{44},$ taking values from $[0,1]$, subject to the condition $\sum_{i=1}^{44}c_{i}\leq 1$ so that $c_{45}=1-\sum_{i=1}^{44}c_{i}$. Now, running all $44$ loops consecutively needs a very high waiting time with a large computational cost, even for an increment of size $ 10^{-1}$ in coefficient values. Hence, to generate a sufficiently large number of sample matrices for a single run, we restrict ourselves in $20$ nested loops, with non-equispaced (unbiased) increments in values approximately of size $\leq 10^{-2},$ each updating values for at least two coefficients together.
Then, we verify the inequality relations as given in Theorem \ref{theorem:ineq} among the coefficients $k_2,k_3$ and $k_4$ of the characteristic polynomials of these matrices.

We emphasize that a computational study for those pairs of permutation matrices is carried out in  \cite{harlev}  to determine the boundaries of $\omega_n$ for $n=2, \ldots, 11$.  In Figure \ref{fig:3}, we show the approximate region for $\omega_5^0$ by computing eigenvalues of several trace zero DS matrices using MATLAB. In support of the conjecture given in \cite{harlev}, we also observe that the outermost eigenvalue paths, which form the boundaries of $\omega_5^0,$ arise from the convex combinations of pairs of trace zero permutation matrices.

Clearly, the quadrangular-shaped region, connecting four extreme points $\mathbf{a, b, c, d}$ on the circle, 
i.e., $\Pi_5^0$, is in $\omega_5^0.$ The line joining $\mathbf{e,f},$ i.e., $\Pi_3^0$ is a part of $\omega_5^0.$ In this context, note that $[-1,1]\cup\Pi_3^0\cup \Pi_5^0 \subsetneq \omega_5^0$ \cite{Benvenuti2015}. The non-real eigenvalues of $A(t)=tP_{(12345)}+(1-t)P_{(124)(35)},$ for $0\leq t \leq 0.2820\, \mbox{(approximately)},$ form the curves $\mathbf{eg}$ and $\mathbf{fh}.$ Besides, for some $t$ where $t \in [0, 1],$ non-real eigenvalues of $tP_{(12345)} + (1-t)P_{(14253)}, tP_{(12345)} + (1-t)P_{(15432)},$ and $tP_{(123)(45)} + (1-t)P_{(132)(45)},$ and real eigenvalues of $tP_{(12345)} + (1-t)P_{(14)(235)},$  form the part of the boundary arcs joining $\mathbf{a,b}$ and $\mathbf{c,d};$ $\mathbf{a,d}$ and $\mathbf{b,c}; \mathbf{e,f};$ and $\mathbf{i,j},$ respectively.
 Thus, the region, bounded by all the solid blue curves and the horizontal line segment $\mathbf{ij}\equiv[-1,1],$ is contained in $\omega_5^0.$ 
 This kind of realization of the boundary arcs, by single parameterized matrices, of the region of eigenvalues of all stochastic matrices of any order $n$, can be seen in  \cite{johnson2017matricial}.
  It is to be noted that the permutation matrices $P_{(12345)}$ and $P_{(124)(35)}$ are mentioned in case $(iii)(B)$ corresponding to $k_4^{min},$ in the proof of Theorem \ref{theorem:ineq}.

\begin{figure}[H]
     \centering
     \includegraphics[height=7 cm, width=7.5 cm]{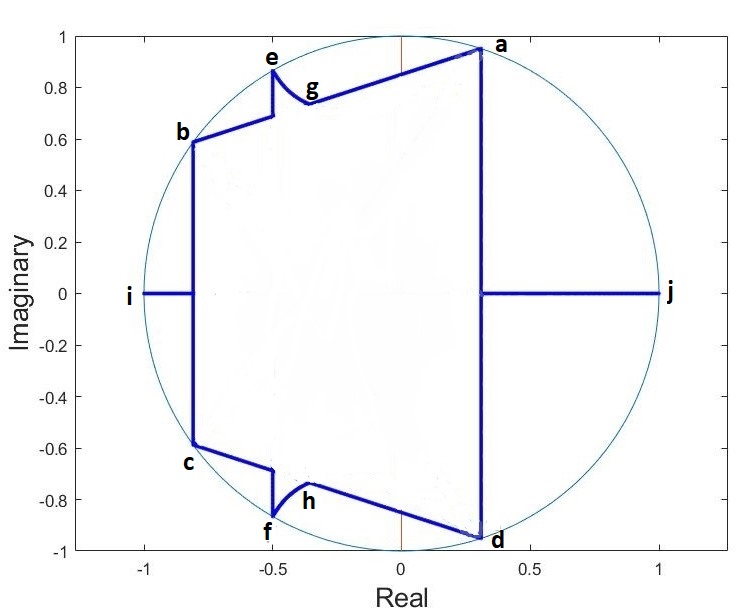}
     \caption{The approximate region for $\omega_5^0.$ The region bounded by the solid blue curve together with the horizontal line estimates $\omega_5^0.$ 
     }\label{fig:pds} \label{fig:3}
 \end{figure}

\noindent{\textbf{Conclusion.}} We derive the trace of the product of two permutation matrices of order $5$ through non-isomorphic weighted digraphs corresponding to the permutation matrices. Employing this result, we determine the explicit forms in terms of the permutation matrices whose convex combination defines a trace zero DS matrix $A$ of order $5$ such that $A^k$ has zero trace, when $k\in\{2,3,4,5\}$. Then, we derive certain necessary conditions for the coefficients of a generic monic polynomial of degree $5$ such that it can be realized as the characteristic polynomial of some trace zero DS matrix of order $5$. Finally, using these results, we approximate the eigenvalue region of trace-zero DS matrices of order $5.$\\


\noindent{\textbf{Acknowledgement.}} The authors thank M. Rajesh Kannan for his insightful initial discussions on this paper. AM thanks IIT Kharagpur for financial support when this work was done.




\begin{thebibliography}{}

\bibitem{Benvenuti2015} L. Benvenuti. A note on eigenvalues location for trace zero doubly stochastic matrices. \textit{Electron.
J. Linear Algebra}, 30:599–604, 2015.

\bibitem{Benvenuti} Luca Benvenuti. Personal communication.

\bibitem{Benvenuti2018} Luca Benvenuti. The NIEP for four dimensional Leslie and doubly stochastic matrices with zero trace from the coefficients of the characteristic polynomial. \textit{Linear Algebra Appl.}, 544:286–298, 2018.

\bibitem{dummit1991} David S Dummit and Richard M Foote. Abstract algebra, volume 1999. Prentice Hall Englewood
Cliffs, NJ, 1991.

\bibitem{harlev} Amit Harlev, Charles R Johnson, and Derek Lim. The doubly stochastic single eigenvalue
problem: A computational approach. \textit{Experimental Mathematics}, pages 1–10, 2020.

\bibitem{johnson1981row} Charles R Johnson. Row stochastic matrices similar to doubly stochastic matrices. \textit{Linear and Multilinear Algebra}, 10(2):113–130, 1981.

\bibitem{johnson2017matricial} Charles R. Johnson, and Pietro Paparella, A matricial view of the Karpelevi{\v{c}} Theorem. \textit{Linear Algebra and its Applications}, 520:1--15, 2017.
  

\bibitem{Johnson2018} Charles R. Johnson, Carlos Mariju{\'a}n, Pietro Paparella, and Miriam Pisonero. The NIEP. \textit{In Operator theory, operator algebras, and matrix theory, volume 267 of Oper. Theory Adv. Appl.}, pages 199–220. Birkh\"{a}user/Springer, Cham, 2018.

\bibitem{kalman2000matrix} Dan Kalman. A matrix proof of newton’s identities. \textit{Mathematics Magazine}, 73(4):313–315, 2000.

\bibitem{trace0} Thomas J. Laffey and Eleanor Meehan. A characterization of trace zero nonnegative 5 × 5 ma-
trices. \textit{Linear Algebra Appl.}, 302/303:295–302, 1999. Special issue dedicated to Hans Schneider
(Madison, WI, 1998).

\bibitem{Lei} Ying-Jie Lei, Wei-Ru Xu, Yong Lu, Yan-Ru Niu, and Xian-Ming Gu. On the symmetric doubly
stochastic inverse eigenvalue problem. \textit{Linear Algebra Appl.}, 445:181–205, 2014.

\bibitem{Levick2015} Jeremy Levick, Rajesh Pereira, and David W. Kribs. The four-dimensional Perfect-Mirsky
Conjecture. \textit{Proc. Amer. Math. Soc.}, 143(5):1951–1956, 2015.

\bibitem{mandal2019eigenvalue} Amrita Mandal, Bibhas Adhikari, and M Rajesh Kannan. On the eigenvalue region of permutative doubly stochastic matrices, accepted in the \textit{Bulletin of the Iranian Mathematical Society}, 
  \textit{arXiv preprint arXiv:1910.01829}, 2019.

\bibitem{marijuan2023note} C Mariju{\'a}n and M Pisonero. A note for the sniep in size 5. \textit{Linear Algebra and its Applications}, 2023.

\bibitem{Mashreghi2007} Javad Mashreghi and Roland Rivard. On a conjecture about the eigenvalues of doubly stochastic matrices. \textit{Linear Multilinear Algebra}, 55(5):491–498, 2007.

\bibitem{Mourad} Bassam Mourad, Hassan Abbas, Ayman Mourad, Ahmad Ghaddar, and Issam Kaddoura. An
algorithm for constructing doubly stochastic matrices for the inverse eigenvalue problem. \textit{Linear
Algebra Appl.}, 439(5):1382–1400, 2013.

\bibitem{Nader} Rafic Nader, Bassam Mourad, Alain Bretto, and Hassan Abbas. A note on the real inverse
spectral problem for doubly stochastic matrices. \textit{Linear Algebra Appl.}, 569:206–240, 2019.

\bibitem{Perfect1965} Hazel Perfect and L. Mirsky. Spectral properties of doubly-stochastic matrices. \textit{Monatsh. Math.}, 69:35–57, 1965.

\bibitem{niepebl} J. Torre-Mayo, M. R. Abril-Raymundo, E. Alarcia-Est\'{e}vez, C. Mariju{\'a}n, and M. Pisonero. The nonnegative inverse eigenvalue problem from the coefficients of the characteristic polynomial.
EBL digraphs. \textit{Linear Algebra Appl.}, 426(2-3):729–773, 2007.

\bibitem{wanggoldbach} Yuan Wang. Goldbach conjecture, volume 4. World scientific, 2002.

\end{thebibliography}

 \noindent{\textbf{Appendix}}\\
For a comprehensive understanding, we provide a detailed calculation for the proof of the inequality relations outlined in $(c)(ii)$ of Theorem \ref{theorem:ineq}, as follows.\\
Given that $k_3^{min}=-5c(1-c)^2$ and $k_4^{min}=-5c^3(1-c),$ where $0<c<1,$
let $\phi(k_4)=k_4-f\left(c\right),$ where $f(c)=-5c^3(1-c)$ and $c=\frac{1}{2}(-k+\sqrt{k^2+4k})$ is a function of $k_4$ only with ${k}=\frac{k_4}{k_3},$ $k_3\neq 0.$ Then $\phi'(k_4)=1-f'\left(c\right)c'(k_4)=
1+5c^2(3-4c)c'(k_4),$ where $c'(k_4)=\frac{1}{2k_3}\left(\frac{k+2}{\sqrt{k^2+4k}}-1\right)\leq 0.$ Hence it is obvious that $\phi'(k_4)\geq 0$ for $-3+4c \geq 0$ or equivalently $\phi'(k_4)\geq 0$ for $k\geq \frac{9}{4}.$ Thus we obtain $\phi(k_4) \geq \phi(k_4^{min})=0$ for $k\geq \frac{9}{4}.$ 
Similarly, let $\psi(k_3)=k_3-g\left(c\right),$ where $g(c)=-5c(1-c)^2$ and $c=\frac{1}{2}(-k+\sqrt{k^2+4k})$ is a function of $k_3$ only with ${k}=\frac{k_4}{k_3},$ $k_3\neq 0.$ Then $\psi'(k_3)=1-g'\left(c\right)c'(k_3)=1+5(1-c)(1-3c)c'(k_3),$ where $c'(k_3)=-\frac{k_4}{2k_3^2}\left(\frac{k+2}{\sqrt{k^2+4k}}-1\right)\geq 0.$ Hence it is obvious that $\psi'(k_3)\geq 0$ for $1-3c \geq 0$ or equivalently $\psi'(k_3)\geq 0$ for $k\leq \frac{1}{6}.$ Thus we get $\psi(k_3) \geq \psi(k_3^{min})=0,$ for $k\leq \frac{1}{6}.$ 

Now, $k_3-f(c) \geq k_3^{min}-f(c)=5c(1-c)(c^2+c-1),$ where $c=\frac{1}{2}(-k+\sqrt{k^2+4k})$ and ${k}=\frac{k_4}{k_3},$ $k_3\neq 0.$  Then $k_3-f(c)\geq 0$ for $\frac{\sqrt 5-1}{2}<c<1$ or equivalently we can say $k_3-f(c)\geq 0$ for $1<k<\infty.$ Similarly $k_4-g(c) \geq k_4^{min}-g(c)=-5c(1-c)(c^2+c-1),$ where $c=\frac{1}{2}(-k+\sqrt{k^2+4k})$ and ${k}=\frac{k_4}{k_3},$ $k_3\neq 0.$  Then $k_4-g(c)\geq 0$ for $0<c< \frac{\sqrt 5-1}{2}$ or equivalently we can say $k_4-g(c)\geq 0$ for $k\leq 1.$
Thus, combining all the above results, we get the inequality relations as provided in statement $(c)(ii)$ of Theorem \ref{theorem:ineq}.

\end{document}